\documentclass{IEEECDC}

\usepackage[left=0.75in,right=0.75in,top=0.73in,bottom=.76in]{geometry}
\usepackage{quiver}
\usepackage{verbatim}
\usepackage{xcolor}
\let\proof\relax
\let\endproof\relax
\usepackage{amsmath, amsfonts, amssymb, amsthm, mathtools}
\usepackage{thmtools}
\usepackage[frozencache,cachedir=minted-cache]{minted}
\usepackage{hyperref}
\newcommand{\red}[1]{\textcolor{red}{#1}}

\newcommand{\mh}[1]{\textcolor{green}{MH: #1}}
\newcommand{\tyler}[1]{\textcolor{purple}{TH: #1}}

\newtheorem{theorem}{Theorem}[section]

\theoremstyle{definition}
\newtheorem{definition}[theorem]{Definition}
\newtheorem{example}[theorem]{Example}
\newtheorem{remark}[theorem]{Remark}

\newtheorem{problem}{Problem}

\makeatletter
\newcommand{\lrarrow}{\mathrel{\mathpalette\lrarrow@\relax}}
\newcommand{\lrarrow@}[2]{%
  \vcenter{\hbox{\ooalign{%
    $\m@th#1\mkern6mu\rightarrow$\cr
    \noalign{\vskip1pt}
    $\m@th#1\leftarrow\mkern6mu$\cr
  }}}%
}
\newcommand{\hto}{\ensuremath{\,\mathaccent\shortmid\rightarrow\,}}
\newcommand{\maps}{\colon}
\newcommand{\R}{\mathbb{R}}
\DeclareMathOperator{\Para}{Para}
\newcommand{\Conv}{\textnormal{\textbf{Conv}}}
\newcommand{\id}{\mathrm{id}}
\newcommand{\define}[1]{\textbf{#1}}

\definecolor{mygray}{RGB}{232, 234, 237}

\makeatother

\begin{document}
\title{Modeling Model Predictive Control: A Category Theoretic Framework for Multistage Control Problems}
\author{Tyler Hanks, Baike She, Evan Patterson, Matthew Hale, Matthew Klawonn, James Fairbanks*
\thanks{*Tyler Hanks and James Fairbanks are with the Department of Computer Science, University of Florida, emails: \texttt{\{t.hanks,fairbanksj\}@ufl.edu}. Baike She and Matthew Hale are with the Department of Electrical and Computer Engineering, Georgia Tech, emails: \texttt{\{bshe6,mhale30\}@gatech.edu}. Evan Patterson is with the Topos Institute, email: \texttt{evan@epatters.org}. Matthew Klawonn is with the Information Directorate, U.S. Air Force Research Lab, email: \texttt{matthew.klawonn.2@us.af.mil}.}
\thanks{Tyler Hanks was supported by the National Science Foundation Graduate
Research Fellowship Program under Grant No. DGE-1842473. Any opinions, findings,
and conclusions or recommendations expressed in this material are those of the author(s)
and do not necessarily reflect the views of the NSF. 
Matthew Hale and Baike She were supported by the Air Force Office of
Scientific Research (AFOSR) under Grant No. FA9550-23-1-0120. 
 Matthew Klawonn was funded by a SMART SEED grant. James Fairbanks was supported by DARPA under award no. HR00112220038.}}

%
\maketitle

\begin{abstract}

Model predictive control (MPC) is an optimal control technique which involves solving a sequence of constrained optimization problems across a given time horizon. 
In this paper, we introduce a category theoretic framework for constructing complex MPC problem formulations by composing subproblems.
Specifically, we construct a monoidal category {\textemdash} called $\Para(\Conv)$ {\textemdash} whose objects are Euclidean spaces and whose morphisms represent constrained convex optimization problems. 
We then show that the multistage structure of typical MPC problems arises from sequential composition in $\Para(\Conv)$, while parallel composition can be used to model constraints across multiple stages of the prediction horizon.
This framework comes equipped with a rigorous, diagrammatic syntax, allowing for easy visualization and modification of complex problems. 
Finally, we show how this framework allows a simple software realization in the Julia programming language by integrating with existing mathematical programming libraries to provide high-level, graphical abstractions for MPC.
\end{abstract}

\section{Introduction}

Model-predictive control (MPC) originated in process control in chemical engineering~\cite{qin1997overview}, and
it has found success in many applications, including 
autonomous driving \cite{kabzan2019learning}, battery charging \cite{hredzak2013model}, path planning \cite{luis2020online}, and energy systems \cite{mariano2021review}. MPC consists of specifying and repeatedly solving constrained optimization problems designed to model the response of a system to control inputs while minimizing a cost function over a finite prediction horizon. 
The generality of MPC gives it several benefits, including the ability to incorporate constraints on a system's state and control inputs that must be satisfied at each timestep or across multiple different timesteps in the prediction horizon.

As the complexity of MPC problems has grown over time, so too has interest in the identification and use of structure inherent to such problems, such as sparsity or decoupling in a problem's constraints. For example, the multistage structure of classic MPC optimization problems has been identified as a rich source of specialized solution algorithms e.g., \cite{domahidi_efficient_2012}\cite{odonoghue_splitting_2013}. In addition to performance benefits, identification of such underlying structure reveals how complex problems can be built up \emph{compositionally} from component subproblems. This benefits MPC practitioners by allowing them to more easily specify complicated MPC problems in a modular and extensible fashion.

In this vein, we develop a novel perspective on the compositional nature of MPC problems. Specifically, we formalize the multistage structure inherent to MPC problems by utilizing tools from \emph{applied category theory}, a discipline focused on applying abstract mathematics to understand structure within science and engineering problems. Such techniques have been applied, for example, to the study of cyber-physical systems \cite{Bakirtzis_2021}, dynamical systems \cite{myers2022categorical}, and machine learning algorithms \cite{fong2019backprop}. Recently, the ICRA Compositional Robotics workshops have demonstrated interest in applying category theoretic techniques to problems in robotics and autonomy \cite{wilhelm_constraint_nodate}\cite{bakirtzis_categorical_nodate}, and the developments we present contribute to this growing body of work. 

Our specific technical contributions are as follows. We first show that convex bifunctions, which were introduced by Rockafellar \cite{theRock} for encoding general convex programs, form a mathematical structure called a symmetric monoidal category (SMC) that we call $\Conv$. This means that bifunctions can be composed in sequence, in parallel, and hierarchically. 

We then apply the \emph{parameterized category construction} to $\Conv$ to yield a new monoidal category, $\Para(\Conv)$, of \emph{parameterized} convex bifunctions. This formalizes the important semantic distinction between control variables and state variables in MPC problems, namely that the state variables couple across timesteps while the control inputs are typically independent to each timestep. 

Finally, we show how each type of composition in $\Para(\Conv)$ can be used to build complex MPC optimization problems in a modular and automated fashion: sequential composition recovers the multistage structure of MPC problems while parallel composition can encode constraints between different timesteps and hierarchical composition allows collections of subproblems to be ``black boxed" and treated as a single subproblem in a larger composite problem.

In addition to showing that MPC optimization problems are an instance of a general categorical abstraction, framing MPC in the unified mathematical language of category theory comes with a number of practical benefits, including a rigorous graphical syntax for compositionally building MPC problems, and a straightforward path to a clean and correct software implementation. Our prototype Julia implementation of this framework demonstrates these benefits by integrating with existing mathematical programming libraries to provide high-level, graphical abstractions for MPC.

The rest of the paper is structured as follows. Section~\ref{sec:motivation} presents motivation and problem statements. Section~\ref{sec:background} provides requisite background in category theory. Section~\ref{sec:ConvCat} presents convex bifunctions and their organization into a symmetric monoidal category. Section~\ref{sec:MPCinConv} presents the main result, a characterization of convex MPC optimization problems as morphisms in $\Para(\Conv)$. Section~\ref{sec:Impl} presents our Julia implementation, and Section VII concludes.

\section{Motivation and Problem Statements}
\label{sec:motivation}

In this section, we aim to build intuition behind the compositional structure of MPC optimization problems. MPC typically operates on a discretized version of the system dynamics $\dot x = f(x, u)$ represented by $x_{i+1} = f_d(x_i, u_i)$ for $i\in \mathbb{N}$, where $f_d$ can be generated from $f$ using a 
discretization technique such as Euler's method or the Runge-Kutta method. When the dynamics are discretized, the MPC problem naturally exhibits a multistage structure which we discuss below.
\vspace{-2ex}
\subsection{Multistage MPC Problems}
In this paper, we consider multistage MPC problems of the following form.
\begin{definition}
\label{def:mpc}
    Let $X$ and $U$ be Euclidean spaces, i.e., $X=\R^n$ and $U=\R^m$ for some natural numbers $n$ and $m$. Given a discrete dynamical system $f_d \maps X\times U\rightarrow X$, cost function $\ell\maps X\times U\rightarrow \R$, 
    initial state~$x_{init} \in X$,
    and constraint function $g\maps X\times U\rightarrow \R^p$, the \define{receding horizon MPC optimization problem} defined by this data\footnote{The word ``data" in this context refers to the mathematical objects that go into a definition.} is
    \begin{equation*}
    \begin{array}{ll@{}ll}
    \underset{u_k, \ldots, u_{N-1}
    }{\text{minimize}} &  \sum_{i=k}^{N-1}\ell (x_i, u_i) & \\
    \text{subject to} & x_{i+1}=f_d(x_i, u_i), & \ i=k,\dots,N-1, \\
    & g(x_i, u_i) \leq 0, & \ i=k,\dots,N-1, \\
    & x_k = x_{init},
    \end{array}
\end{equation*}
where $k < N$ is the current timestep, $x_{init}$ is the actual state of the system at timestep $k$, and $N$ is the prediction horizon. Similarly, the \emph{moving horizon} MPC optimization problem is the same formulation as above but with all 
instances of ``$N-1$'' replaced by ``$k+N-1$''. The $x_i$'s are the \define{state variables} while the $u_i$'s are the \define{control inputs}. Thus we refer to $X$ as the \define{state space} and $U$ as the \define{control space}. 
\end{definition}
 
 In practice, there is often an additional constraint or cost term on the final state to ensure stability. When $g$ and $\ell$ are convex and $f_d$ is affine, the MPC optimization problem is convex. 
 Such settings are well-studied in MPC literature, and we restrict to this case in the present work for simplicity of presentation and implementation. However, we emphasize that the general modeling framework developed in this paper can also be applied to model non-convex problems.
 
 \begin{example}
 \label{ex:running}
 Suppose we want to implement a \emph{constrained} linear-quadratic regulator for a linear system given by $x_{k+1} = Ax_k + Bu_k$ by minimizing the cost 
 \begin{equation*}
     J = \sum_{k=0}^\infty x_k^TQx_k + u_k^TRu_k,
 \end{equation*}
while remaining within operational bounds $x_{min}\leq x_k \leq x_{max}$ and $u_{min}\leq u_k \leq u_{max}$ for all $k\in [0,\infty)$, where $x_{min}, x_k, x_{max} \in \mathbb{R}^n$, $u_{min},u_k,u_{max} \in \mathbb{R}^m$, and $Q,R$ positive semi-definite matrices of the correct dimension. Although this scenario arises often in practice, such problems (especially as their dimension grows) often do not have an explicit solution or computing an explicit solution is computationally difficult \cite{Alessio2009}. Hence, it is common to choose a prediction horizon $N$ and reformulate the above as a moving horizon MPC problem:

\begin{equation*}
    \begin{array}{ll@{}ll}
    \text{minimize} &  \sum_{i=k}^{k+N-1} x_i^TQx_i + u_i^TRu_i & \\
    \text{subject to} & x_{i+1}=Ax_i + Bu_i, \  i=k,\dots,k+N-1, \\
    & x_{min}\leq x_i \leq x_{max}, \  i=k,\dots,k+N-1, \\
    & u_{min}\leq u_i \leq u_{max}, \  i=k,\dots,k+N-1, \\
    & x_k = x_{init}.
    \end{array}
\end{equation*} 
This problem will serve as a simple, illustrative example for demonstrating the use of our framework. 
 \end{example}

Problem formulations such as these are typically referred to as having multistage or staged structure because they can intuitively be viewed as distinct subproblems linked together by some of the decision variables.

To illustrate this, consider the following parameterized problem for a fixed $i$ representing a single stage of the MPC problem. We call this an \define{MPC subproblem}.
\begin{equation}
    x_i\mapsto \begin{array}{ll@{}ll}
    \text{minimize} & \ell (x_i, u_i) & \\
    \text{subject to} & x_{i+1}=f_d(x_i, u_i), & \\
    & g(x_i, u_i) \leq 0. & \\
    \end{array}
\end{equation}
When viewed this way in terms of individual stages, it is clear that the goal of the $i^{th}$ subproblem is to find an optimal $u_i$ and $x_{i+1}$ given a perturbation to its constraints by $x_i$. Subsequently, the value of $x_{i+1}$ becomes a perturbation to the constraints of the $(i+1)^{th}$ subproblem. This is the multistage structure we need to capture in order to give a compositional perspective on MPC problems. Crucially, only the state variables couple subproblems through time; the control variables are independent to each subproblem. These observations motivate the central problems of this paper. 

\vspace{-2ex}
\subsection{Problem Formulation}
In this work, we aim to compositionally model convex MPC optimization problems (i.e., those with linear dynamics and convex costs) using a category-theoretic framework.
Achieving this entails solving the following problems: 

\begin{problem}
\label{prob:1}
How can we model the composition of convex optimization problems using category theory?
\end{problem}
\begin{problem}
\label{prob:2}
How can we formally account for the semantic distinction between state variables and control variables in these categorical models?
\end{problem}
\begin{problem}
\label{prob:3}
How do MPC optimization problems fit into our model of composition of convex programs? 
\end{problem}
We solve Problems~\ref{prob:1} and \ref{prob:2} in Section \ref{sec:ConvCat} by constructing two new categories of convex optimization problems. Problem \ref{prob:3} is solved in Section \ref{sec:MPCinConv}. The Julia implementation in Section \ref{sec:Impl} shows the practical utility of solving these problems. 


\section{Category Theory Background} \label{sec:background}
This section briefly reviews concepts from category theory used in the rest of this paper. We refer the reader to \cite{fong2019invitation} for a more detailed introduction to applied category theory.

The most basic notion in category theory is that of a category. Categories can be seen as directed graphs with extra structure and are often used to model a specific domain of mathematical objects and relationships between them. Different categories can then be related by functors, which are the structure-preserving maps between categories.

\begin{definition}
\label{def:cat}
    A \define{category} $\mathcal{C}$ consists of a collection of \emph{objects} $X,Y,Z,\dots$, and a collection of \emph{morphisms} (also called \emph{arrows}) $f,g,h, \dots$, together with the following data: 
    \begin{itemize}
        \item Each morphism has a specified \emph{domain} and \emph{codomain} object; the notation $f:X\rightarrow Y$ signifies that $f$ is a morphism with domain $X$ and codomain $Y$.
        \item Each object $X\in \mathcal{C}$ has a distinguished \emph{identity} morphism $\text{id}_X:X\rightarrow X$.
        \item For any triple of objects $X,Y,Z\in \mathcal{C}$ and any pair of morphisms $f:X\rightarrow Y$ and $g:Y\rightarrow Z$, there is a specified \emph{composite} morphism $g\circ f:X\rightarrow Z$.
    \end{itemize}
    This data is subject to the following two axioms:
    \begin{itemize}
        \item (Unitality) For any morphism $f:X\rightarrow Y$, we have $f\circ \text{id}_X = f = \text{id}_Y\circ f$,
        \item (Associativity) For any objects $W,X,Y,Z\in\mathcal{C}$ and morphisms $f:W\rightarrow X$, $g:X\rightarrow Y$, $h:Y\rightarrow Z$, we have $h\circ(g\circ f) = (h\circ g)\circ f$.
    \end{itemize}
The set of all objects of $\mathcal{C}$ is denoted $\text{Ob}(\mathcal{C})$.
\end{definition}


The canonical example of a category is \textbf{Set}, which has sets for objects and functions between sets for morphisms. Composition in this category is 
ordinary function composition and identities are the identity functions, namely $\text{id}_X(x)=x$. 
The next example lends intuition to the composition operation that we will define for the category $\Conv$. The category $\textbf{Rel}$ has sets as objects and has binary relations of the form $R\subseteq X\times Y$ as morphisms from $X$ to $Y$.
The composite of two relations $R\subseteq X\times Y$ and $S\subseteq Y\times Z$ is defined to be
\begin{multline}
\label{eq:relcomp}
    S\circ R \coloneqq \{(x,z)\in X\times Z \mid \textnormal{there exists } y \in Y \\ \textnormal{such that } (x,y) \in R \textnormal{ and } (y,z)\in S \}.
\end{multline}

\noindent
The identity relation is $\text{id}_X = \{(x,x')\in X\times X \mid x = x'\}.$
We can map the objects and morphisms of one category to those of another another using a \textit{functor}, which can be seen as a graph homomorphism that preserves the additional structure of a category, namely composition and identities.

\begin{definition}\label{def:functor}
    Given categories $\mathcal{C}$ and $\mathcal{D}$, a \define{functor} $F:\mathcal{C}\rightarrow \mathcal{D}$ consists of the following maps:
    \begin{itemize}
        \item an \define{object map} $c\mapsto F(c)$ for all $c\in \textnormal{Ob}(\mathcal{C})$,
        \item a \define{morphism map} respecting domains and codomains, i.e., $(f: c \to c') \mapsto (F(f): F(c) \to F(c'))$ for every morphism $f\maps c\to c'$ in $\mathcal{C}$.
    \end{itemize}
    These maps must satisfy the following equations:
    \begin{itemize}
        \item $F(\text{id}_c) = \text{id}_{F(c)}$ (identities are preserved),
        \item $F(g\circ f) = F(g)\circ F(f)$ (composition is preserved),
    \end{itemize}
    for all objects $c$ and composable morphisms $f,g$ in $\mathcal{C}$.
\end{definition}
As an example, there is a functor from $\textbf{Set}$ to $\textbf{Rel}$ that is the identity on objects and sends a function $f\maps X\rightarrow Y$ to its graph, i.e., the relation $\{(x,y)\in X\times Y \mid y = f(x)\}$.

\noindent

\section{Categories of Convex Bifunctions}
\label{sec:ConvCat}
Now that we have established the necessary category-theoretic background, we are ready to solve Problem \ref{prob:1}, namely, modeling the composition of convex programs
using categories. 
According to Definition \ref{def:cat}, morphisms in a category are required to be composable, so our solution to this problem will consist of constructing a category whose morphisms represent convex optimization problems.

\vspace{-2ex}
\subsection{Convex bifunctions}
The framework we develop is based on a formalization of convex optimization problems, called \textit{convex bifunctions}, introduced by Rockafellar \cite[\S29]{theRock}.

\begin{definition}\label{def:bifunction}
    A \define{convex bifunction} $F\maps \R^m\hto \R^n$ is a (jointly) convex function of form $F\maps \R^m \times \R^n \rightarrow \R\cup \{\pm\infty\}$.
    
    
\end{definition}

Henceforth, we denote the extended reals $\R\cup\{\pm\infty\}$ by $\overline{\R}$. Our purpose in introducing convex bifunctions is to represent minimization problems that are 
parameterized in the constraint space. For example, consider a general convex program $(P)$ over a vector of decision variables $x\in \R^n$:
\begin{equation*}
    \begin{array}{ll@{}ll}
    \text{minimize} & f_0 & (x) & \\
    \text{subject to} & f_i & (x) \leq u_i, & i=1,\dots,r \\
    & f_j & (x) = u_j, & j=r+1,\dots,m. \\
    \end{array}
\end{equation*}
Here $f_i$ is a convex function $\R^n\rightarrow \overline{\R}$ for each $i\in\{0,\dots,r\}$ and $f_j$ is an affine function $\R^n\rightarrow\overline{\R}$ for each $j\in\{r+1,\dots,m\}$. Writing $u \coloneqq (u_1,\dots,u_m) \in \R^m$ for the constraint parameters, let $Su$ denote the set of vectors $x \in \R^n$ 
that satisfy the above constraints for that choice of $u$. Different values of $u$ correspond to different parameterizations of the constraints of $(P)$.
We can encapsulate this whole family of optimization problems into a convex bifunction $F\maps \R^m\hto \R^n$ by setting
    $F(u,x) \coloneqq f_0(x) + \delta(x \mid Su)$,
where $\delta(x \mid Su)$ denotes the convex indicator function
\begin{equation*}
    x\mapsto 
    \begin{cases}
        0 & \text{if } x \in Su,\\
        +\infty & \text{otherwise}.
    \end{cases}
\end{equation*}
We call $F$ the convex bifunction associated with the convex program $(P)$. 
\begin{remark}
\label{remark:representation}
In general, any convex program is uniquely determined by its associated bifunction (see 
\cite[\S29]{theRock}). For our purposes, convex bifunctions will be sums where some terms in the sum are indicator functions and others are not. The convex program associated with a given convex bifunction is one whose hard constraints are satisfied when all indicator function terms take the value $0$, implying the program's objective function is given by the non-indicator function terms of the bifunction. 
\end{remark}

It is useful to note that convex bifunctions generalize linear transformations. In other words, there is a canonical way to embed linear transformations as convex bifunctions. For a linear map $A\maps X\rightarrow Y$, the \emph{convex indicator bifunction of the linear transformation} is the bifunction $\delta_A\maps X\hto Y$ with
\begin{equation*}
    \delta_A(x,y) \coloneqq \delta(y \mid Ax) \coloneqq \left\{
    \begin{array}{cr}
         0 & \text{if } y = Ax,  \\
         +\infty & \text{if } y \neq Ax. \\ 
    \end{array}
    \right.
\end{equation*}

Since convex bifunctions can be identified with convex programs, defining a composition operation for bifunctions directly yields a composition operation for convex programs. Given convex bifunctions $F\maps U\hto X$ and $G\maps X \hto Y$, their composite is defined by
\begin{equation} \label{eq:inf_mult}
    (G\circ F)(u,y) \coloneqq \inf_{x\in X}\{F(u,x) + G(x,y)\}, 
\end{equation}
 for all $u\in U$ and $y\in Y$. We refer to this operation as \define{bifunction composition}. We adopt the convention that $\inf f =-\infty$ when $f$ is unbounded below. 
 
 Bifunction composition generalizes composition of linear transformations, i.e., $\delta_{A_2}\circ\delta_{A_1} = \delta_{A_2\circ A_1}$ for composable linear maps $A_1$ and $A_2$. It also generalizes the relational composition given by~\eqref{eq:relcomp} when relations forming a convex subset are encoded by their convex indicator bifunctions. In that case bifunction composition agrees with relation composition.
Crucially, when viewed as representing convex programs, bifunction composition captures the notion that one problem's decision variables parameterize the constraints of the subsequent problem.
\vspace{-2ex}
\subsection{The SMC of Convex Bifunctions}
We next show that bifunction composition in fact
gives rise to a category of convex bifunctions; the construction of this category will solve Problem~\ref{prob:1}. 

\begin{restatable}{theorem}{convCategory}
    There is a category, \emph{\textbf{Conv}}, in which
    \begin{itemize}
        \item objects are Euclidean spaces,
        \item an arrow from $\R^m$ to $\R^n$ is a convex bifunction $\R^m\hto \R^n$,
        \item the identity bifunction on $X$ is
        \begin{equation*}
            \textnormal{id}_X(x,x') \coloneqq
    \begin{cases}
        0, & \text{if } x = x' \\
        +\infty, & \text{otherwise}
    \end{cases}, \qquad x,x'\in X,
        \end{equation*}
        \item the composition of two bifunctions is given by~\eqref{eq:inf_mult}.
    \end{itemize}
\end{restatable}
\proof
    Given a composable pair of convex bifunctions, their composite is again convex (\cite{theRock}, \S38.5), so the composition operation is well-defined. 
    
    Associativity is straightforward in the case that all infima in the resulting composites are attained.
    Furthermore, using the convention that $\infty + (-\infty) = \infty$ to interpret $F(u,x) + G(x,y)$ where necessary, associativity is easily extended to the unbounded case.

    For unitality, consider a bifunction $F\maps \R^m \hto \R^n$. Then we have 
        $(F\circ \id_{\R^m})(u,x) = \inf_{u'\in \R^m}(\id_{\R^m}(u,u') + F(u',x))$.
    Clearly, the infimum is attained for $u'=u$ because otherwise $\id_{\R^m}(u,u')=\infty$. So $(F\circ \id_{\R^m})(u,x) = 0 + F(u,x) = F(u,x)$. A similar argument shows that $F = \id_{\R^n} \circ F$.
\endproof


Composition in $\Conv$ formalizes a notion of sequential coupling between subproblems. To formulate MPC in this context, we also need the ability to compose convex bifunctions in parallel by combining two independent programs into one that is separable. This requires an additional structure on $\Conv$ called a \emph{monoidal product}. Monoidal products are in turn defined in terms of a functor out of a product category. So we first define the product of categories, which generalizes the familiar Cartesian product of sets.

\begin{definition}
\label{def:product}
    Given categories $\mathcal{C}$ and $\mathcal{D}$, their \define{product category}, denoted $\mathcal{C}\times \mathcal{D}$, is a category in which
    \begin{itemize}
        \item objects are pairs $(c,d)$, where $c\in \text{Ob}(\mathcal{C})$ and $d\in \text{Ob}(\mathcal{D})$, i.e., $\text{Ob}(\mathcal{C}\times \mathcal{D}) = \text{Ob}(\mathcal{C}) \times \text{Ob}(\mathcal{D})$,
        \item a morphism $(c_1,d_1)\rightarrow (c_2,d_2)$ is a pair $(f,g)$, where $f\maps c_1\rightarrow c_2$ is a morphism in $\mathcal{C}$ and $g\maps d_1\rightarrow d_2$ is a morphism in $\mathcal{D}$,
        \item composition and identities are defined componentwise:
        \begin{equation*}
            (f_2,g_2)\circ (f_1,g_1) \coloneqq (f_2\circ f_1, g_2\circ g_1),
            \ 
            \id_{(c,d)} \coloneqq (\id_c,\id_d).
        \end{equation*}
    \end{itemize}
\end{definition}
\noindent
\vspace{-5ex}
\begin{definition} \label{def:monoidal}
    A \define{strict monoidal category} is a category $\mathcal{C}$ equipped with a functor $\otimes:\mathcal{C}\times \mathcal{C}\rightarrow \mathcal{C}$, called the \define{monoidal product}, and a distinguished object $I\in \textnormal{Ob}(\mathcal{C})$, called the \define{monoidal unit}, such that for all objects $A,B,C$ and morphisms $f,g,h$ in $\mathcal{C}$, the following equations hold:
    \begin{itemize}
        \item $(A\otimes B)\otimes C = A\otimes(B\otimes C)$,
        \item $A\otimes I = A$,
        \item $I\otimes A = A$,
        \item $(f\otimes g)\otimes h = f\otimes(g\otimes h)$,
        \item $\text{id}_I \otimes f = f$,
        \item $f \otimes \text{id}_I = f$.
    \end{itemize}
\end{definition}
Monoidal categories are typically written as a triple comprising the category, its monoidal product, and its unit, e.g., $(\mathcal{C},\otimes, I)$. If for every pair of objects $A,B$ in a monoidal category, $A\otimes B\cong B\otimes A$, then it is a \define{symmetric}\footnote{The swap maps $\sigma_{A,B}\maps A\otimes B\cong B\otimes A$ must also obey several coherence conditions which we omit for brevity.} monoidal category. Monoidal categories formalize contexts which admit both (i) sequential 
combinations of arrows through composition and (ii) parallel combinations of arrows through the monoidal product. The next theorem shows how to make $\Conv$ into a monoidal category.

\begin{restatable}{theorem}{convSMC} \label{thm:convSMC}
    Let $\oplus$ be the functor from $\Conv\times\Conv$ to $\Conv$ which takes pairs of Euclidean spaces $(\R^n,\R^m)$ to their direct sum $\R^n\oplus \R^m$ and takes bifunction pairs of the form $(F\maps \R^m\hto \R^n, G\maps \R^p\hto \R^q)$ to their pointwise sum, i.e., $F\oplus G\maps\R^m\oplus\R^p\hto\R^n\oplus\R^q$, where 
    \begin{multline*}
    (F\oplus G)((u,y),(x,z)) = 
        F(u,x) + G(y,z) \\ \text{  for all } u\in\R^m,x\in\R^n,y\in\R^p, \textnormal{and } z\in\R^q.
    \end{multline*}
    Then $(\Conv, \oplus, \R^0)$ is a strict symmetric monoidal category.
\end{restatable}
\begin{proof}
    To show that $\oplus\maps \textbf{Conv}\times \textbf{Conv}\rightarrow \textbf{Conv}$ is a functor, first note that identities are preserved since
    \begin{align*}
            (\id_{\R^n} &\oplus \id_{\R^m})((x,y),(x',y')) \\
             &= \left\{
    \begin{array}{lr}
         0 & \text{if } x = x'  \\
         +\infty & \text{otherwise} \\ 
    \end{array}
    \right.\,\, +\,\, \left\{
    \begin{array}{lr}
         0 & \text{if } y = y'  \\
         +\infty & \text{otherwise} \\ 
    \end{array}
    \right.\\
     &= \left\{
    \begin{array}{lr}
         0 & \text{if } x = x' \textnormal{ and } y = y'  \\
         +\infty & \text{otherwise} \\ 
    \end{array}
    \right.\\
     &= \id_{\R^n\oplus\R^m}((x,y),(x',y')).
    \end{align*}
    To see that $\oplus$ preserves composition, we need to show that 
    \begin{equation*}
        (F_2\circ F_1) \oplus (G_2\circ G_1) = (F_2\oplus G_2)\circ(F_1\oplus G_1)
    \end{equation*}
    for composable morphisms $F_1,F_2,G_1,G_2$ in $\Conv$. We have
    \begin{multline*}
           ((F_2\circ F_1) \oplus (G_2\circ G_1))((x_1,y_1),(x_2,y_2)) \\
            = \inf_{x'}(F_1(x_1,x')+F_2(x',x_2)) \\
            + \inf_{y'}(G_1(y_1,y') + G_2(y',y_2)) \\
            = \inf_{x',y'}\Big(F_1(x_1,x')+G_1(y_1,y') \\
            \hspace{33mm}+F_2(x',x_2)+G_2(y',y_2)\Big) \\
            = ((F_2\oplus G_2)\circ (F_1\oplus G_1))((x_1,y_1),(x_2,y_2)),        
    \end{multline*}
    as desired. The associativity and unitality conditions of $\oplus$ required by Definition \ref{def:monoidal} follow readily from the basic rules of arithmetic. The symmetric structure follows from the isomorphisms $\R^n\oplus \R^m\cong \R^m\oplus \R^n$ for all $n,m\in \mathbb{N}$.
\end{proof}

One benefit of monoidal categories is that they possess an intuitive yet formal graphical calculus called \emph{string diagrams} \cite{Selinger_2010} to represent composite morphisms. In a string diagram, strings (or wires) represent objects and boxes represent morphisms. For example, suppose $f\maps U\hto W\oplus X$, $g\maps X\hto Y$, and $h\maps W\oplus Y\hto Z$ are convex bifunctions in $(\Conv,\oplus,\R^0)$. Then the string diagram

\begin{center}
\includegraphics[scale=.3]{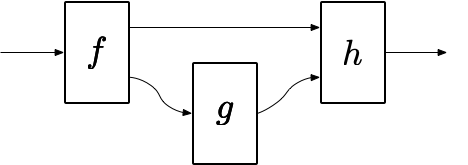}
\end{center}
represents the composite morphism $h\circ (\text{id}_W \oplus g)\circ f$. In this composite, the decision variables of $f$ perturb the constraints of $g$ and $h$, while $h$'s constraints are also perturbed by $g$'s decision variables.

\vspace{-2ex}
\subsection{Parameterized convex bifunctions}

The category \textbf{Conv} provides a formal setting for composing optimization problems such that one problem's decisions parameterize another problem's constraints. MPC optimization problems exhibit this compositional structure, as the state variables at time-step $k$ are computed by the $k^{th}$ MPC subproblem, 
while they are fixed from the perspective of the~$(k+1)^{th}$ subproblem and
serve to parameterize the constraints of the $(k+1)^{th}$ subproblem. However, the setting of $\Conv$ leaves out the important role played by the control variables. These are decision variables which are independent for each subproblem, i.e., they do not couple across time like the state does. 
As stated, $\Conv$ can model the coupling between state variables in MPC but cannot model this independence of the control variables.
To capture this semantic distinction, we review another categorical construction called a \emph{parameterized category} \cite{lenslearn} and then apply it to $\Conv$ to produce the desired results.

\begin{definition}
\label{def:para}
    If $(\mathcal{C},\otimes, I)$ is a strict monoidal category, then the \define{parameterized category} $\Para(\mathcal{C})$ is a category in which:
    \begin{itemize}
        \item the objects are those of $\mathcal{C}$,
        \item a morphism $X\rightarrow Y$ is a pair $(U,f)$ where $U$ is an object of $\mathcal{C}$ and $f\maps U\otimes X\rightarrow Y$ is a morphism in $\mathcal{C}$,
        \item the identity on $X$ is $(I,\id_X)$,
        \item the composite of $(U_1,f)\maps X\rightarrow Y$ and $(U_2,g)\maps Y\rightarrow Z$ is the pair $\big(U_2\otimes U_1, g\circ (\id_{U_2}\otimes f)\big)$.
    \end{itemize}
    %
\end{definition}

Since $\Conv$ is a monoidal category by Theorem \ref{thm:convSMC}, we can directly apply the Para construction to get a new category $\Para(\Conv)$ of \emph{parameterized} convex bifunctions. 
Explicitly, objects in $\Para(\Conv)$ remain Euclidean spaces, while morphisms are pairs of the form $(U, F\maps U\oplus X\hto Y)$ for Euclidean spaces $U$, $X$, and $Y$.
The parameter space $U$ will be used to model the control inputs to the MPC optimization problem.

To better understand the definition of composites in $\Para(\Conv)$, we can visualize them in terms of string diagrams as in \cite{lenslearn}. We draw a parameterized bifunction as a string diagram where the domain and codomain objects are as usual, but the parameter object is represented by a vertical arrow coming into the morphism from above. This serves to make the visual representation of parameters distinct from a morphism's domain and codomain. For example, the morphism $(U_1, f)\maps X\rightarrow Y$ in $\Para(\Conv)$ is drawn as
\begin{center}
\includegraphics[scale=.3]{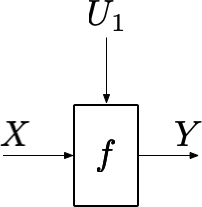}.
\end{center}
Then composites are drawn simply as
\begin{center}
\includegraphics[scale=.3]{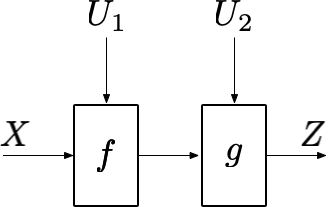}.
\end{center}
\vspace{-3ex}
\section {MPC Optimization Problems are Morphisms}
\label{sec:MPCinConv}

With the ability to compose convex optimization problems according to the rules of \textbf{Conv} and the necessary semantic distinction between control variables and state variables formalized in $\Para(\Conv)$, we are ready to describe the categorical structure of the MPC optimization problem. Specifically, given a convex MPC problem $(Q)$ of the form in Definition \ref{def:mpc}, there exists a morphism $(U,G_Q)$ in $\Para(\Conv)$ such that the optimization problem associated with $Q$ for a prediction horizon of $N$ is the $N$-fold composite $\underbrace{(U,G_Q)\circ\dots \circ (U,G_Q)}_{N \text{ times}}$ which we denote $(U,G_{Q})^N$. 

Recall that at a single time-step $i$ in the prediction horizon,  we have the one-step MPC subproblem
\begin{equation}
\label{eq:onestepprob}
    \begin{array}{ll@{}ll}
    \text{minimize} &  \ell (x_i, u_i) & \\
    \text{subject to} & x_{i+1}=f_d(x_i, u_i) \\
    & g(x_i, u_i) \leq 0.
    \end{array}
\end{equation}
This problem can be interpreted as optimizing over the control input at a single time-step, i.e.,
with a prediction horizon of length~$N=1$. We can reformulate this single timestep problem into a bifunction.

\begin{definition}
\label{def:onestep}
    Given a discrete dynamical system $f_d \maps X\times U\rightarrow X$, cost function $\ell\maps X\times U\rightarrow \R$, and constraint function $g\maps X\times U\rightarrow \R^m$, the \emph{one-step bifunction} representing the one-step subproblem (\ref{eq:onestepprob}) is $G\maps U\oplus X\hto X$, where
    \begin{multline}    
        G((u, x), x') \coloneqq \ell(x,u) + \delta\Bigl(x' \mid f_d(x,u)\Bigr) \\
         + \delta(x,u \mid g(x,u) \leq 0).
    \end{multline}
    Here, $\delta(x,u \mid g(x,u) \leq 0)$ denotes the indicator function of the set $\{(x,u)\in X\times U \mid g(x,u)\leq 0\}$.
\end{definition}
We next show that, under suitable convexity conditions, the one-step bifunction associated with an MPC
subproblem can be regarded as belonging to $\Para(\Conv)$. 

\begin{restatable}{proposition}{genconv}
\label{prop:genconv}
    If $f_d\maps X\times U\rightarrow X$ is affine and the functions $\ell\maps X\times U\rightarrow \R$ and $g\maps X\times U\rightarrow \R^m$ are convex, 
    then the pair $(U,G\maps U\oplus X\hto X)$, where $G$ is the one-step bifunction of the MPC subproblem associated with $f_d,\ell,$ and $g$, is a morphism from $X$ to $X$ in \emph{$\Para(\Conv)$}.
\end{restatable}
\begin{proof}
    Recall that by Definition \ref{def:para}, a morphism $X\rightarrow X$ in $\Para(\Conv)$ is a pair $(U', F\maps U'\oplus X\hto X)$ for some choice of control space $U'$ and bifunction $F$. Choosing $U$ as the control space and $G$ as the bifunction, $(U,G)$ has the correct signature of a morphism $X\rightarrow X$ in $\Para(\Conv)$. Convexity of~$G$ follows immediately from being the sum of a convex function with indicators of convex sets.
\end{proof}

Now that we have an endomorphism\footnote{Endormorphisms are morphisms whose domain and codomain are equal.}
$(U,G)\maps X\rightarrow X$ in $\Para(\Conv)$, we can compose $(U,G)$ with itself an arbitrary number of times. We use this property to formally express the compositional structure of MPC optimization problems, thus solving Problem \ref{prob:3}.
 
\begin{restatable}{theorem}{nfoldcomp}
\label{thm:nfoldcomp}
    Let $\ell\maps X\times U\rightarrow\R$ and $g\maps X\times U\rightarrow \R^m$ be convex functions and let $f_d\maps X\times U\rightarrow X$ be an affine function. If $(U,G)\maps X\rightarrow X$ is the one-step bifunction in $\Para(\Conv)$ generated by these functions, then $(U,G)^N$ is the parameterized convex bifunction representation of the MPC optimization problem associated with $f_d, \ell$, and $g$ over a prediction horizon of $N$.
\end{restatable}
\begin{proof}
    This proof proceeds by induction on $N$. The base case of $N=1$ is proven by Proposition \ref{prop:genconv}. For the inductive step, suppose $(U,G)^k$ is the parameterized convex bifunction representation of the MPC optimization problem associated with $f_d,\ell,$ and $g$ over a prediction horizon of $k$. We need to show that $(U,G)^{k+1} = (U,G)\circ (U,G)^k$ is the parameterized convex bifunction representation of the MPC problem over a prediction horizon of $k+1$. By the inductive hypothesis, we have that
    \begin{multline}
    \label{IH}
    (U,G)^k((u_{k-1},\dots,u_0,x_0),x_k) = \\
    \inf_{x_1,\dots,x_{k-1}}\Big[\sum_{i=0}^{k-1}\ell(x_i,u_i)+\delta\Bigl(x_{i+1} \mid f_d(x_i,u_i)\Bigr)\\ + \delta\Bigl(x_i,u_i \mid g(x_i,u_i)\leq 0\Bigr)\Big].
    \end{multline}
    Then, by the definition of bifunction composition, we have
    \begin{multline}
    \label{res}
    (U,G)\circ(U,G)^k((u_k,\dots,u_0,x_0),x_{k+1}) = \\
    \inf_{x_k}\Big[\ell(x_k,u_k)+\delta\Bigl(x_{k+1} \mid f_d(x_k,u_k)\Bigr)\\ + \delta\Bigl(x_k,u_k \mid g(x_k,u_k)\leq 0\Bigr) + \text{ RHS of~\eqref{IH}}\Big] \\
    = \inf_{x_1,\dots,x_k}\Big[\sum_{i=0}^{k}\ell(x_i,u_i)+\delta\Bigl(x_{i+1} \mid f_d(x_i,u_i)\Bigr)\\ + \delta\Bigl(x_i,u_i \mid g(x_i,u_i)\leq 0\Bigr) \Big],
    \end{multline}
    which is the parameterized convex bifunction representation of the MPC problem associated with $f_d,\ell,$ and $g$ over a prediction horizon of $k+1$, as desired.
\end{proof}

The theorem shows that any MPC optimization problem of the form in Definition \ref{def:mpc} is represented by repeated composition of an endomorphism in $\Para(\Conv)$ to the desired prediction horizon. Our software implementation takes advantage of this result to ensure that the convex program generated from a specification of the dynamics, one-step cost, and one-step constraints, is correct by construction.

\vspace{-2ex}
\subsection{Modeling initial and terminal constraints and costs}

The formulation so far of MPC optimization problems does not impose any 
restrictions in terms of costs or constraints on the state achieved at the end of the time horizon, though such 
restrictions often arise in practice. 
For example, we may wish to modify the objective function of Example \ref{ex:running} to include a terminal cost term $x_N^TQ_Nx_N$ or add the constraint that $x_N = c$ for some desired target vector $c$. Terminal costs and constraints can be represented by a morphism $(\R^0 ,T)\maps X\rightarrow \R^0$ in $\Para(\Conv)$ for some bifunction $T\maps \R^0\oplus X\hto \R^0$. Since $T$ has trivial control space and codomain, it is equivalent to a convex function $X\rightarrow \overline{\R}$. This is precisely the signature needed to encode a terminal cost or constraint function. We can incoporate this bifunction into our problem by post-composition, i.e.,
    $(\R^0,T)\circ (U,G)^N$.

Similarly, fixing the initial state to a value $x_0$ is accomplished by pre-composition with a morphism $(\R^0,\delta_{x_0})\maps \R^0\rightarrow X$ where
\begin{equation*}
\delta_{x_0}(x)\coloneqq \left\{
    \begin{array}{cr}
         0 & \text{if } x = x_0,  \\
         +\infty & \text{if } x \neq x_0. \\ 
    \end{array}
    \right.
\end{equation*}
The formulation of MPC with initial conditions and required terminal state or cost is thus treated symmetrically by our formalism. The general case is $(\R^0, T) \circ (U,G)^N \circ (\R^0, x_0)$, where composition in $\Para(\Conv)$ correctly models the coupling between initial, sequential, and terminal steps of the control process. This is represented graphically as
\begin{center}
    \includegraphics[scale=.3]{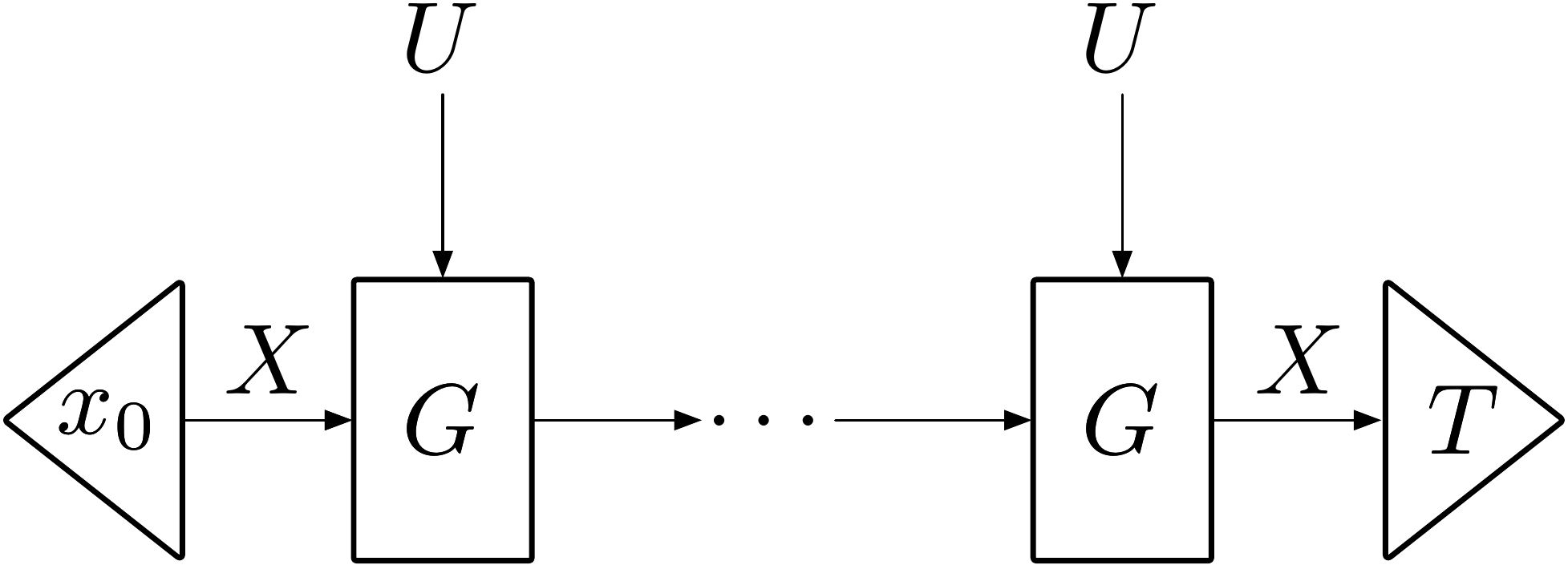},
\end{center}
where opening and closing triangles represent morphisms from $\R^0$ (no inputs) and to $\R^0$ (no outputs) respectively.

\vspace{-2ex}
\subsection{Modeling constraints across timesteps}

Many robotic control applications involve constraints on the state among problem stages separated by more than one timestep. This can be used to encode periodicity constraints, which are useful in applications such as robotic walking, where the position of a robot's leg must periodically return to the same state \cite{yang_equality_2021}. Such constraints can easily be integrated into our framework.

First note that there are natural bifunction encodings of duplicating and merging variables.
\begin{itemize}
    \item Given a Euclidean space $X$, the duplicate map $\Delta_X\maps X\hto X\oplus X$ is given by the bifunction
    \begin{equation*}
        \Delta_X(x, (x_1,x_2))\coloneqq \left\{
    \begin{array}{cr}
         0 & \text{if } x = x_1 = x_2,  \\
         +\infty & \text{otherwise.} \\ 
    \end{array}
    \right.
    \end{equation*}
    \item Given a Euclidean space $X$, the merge map $\mu_X\maps X\oplus X\hto X$ is given by the bifunction
    \begin{equation*}
        \mu_X((x_1,x_2),x)\coloneqq \left\{
    \begin{array}{cr}
         0 & \text{if } x = x_1 = x_2,  \\
         +\infty & \text{otherwise.} \\
         \end{array}
         \right.
    \end{equation*}
\end{itemize}
These copy and merge maps are represented in string diagrams as follows, where the copy map is on the left and the merge map is on the right.
\begin{center}
\includegraphics[scale=.3]{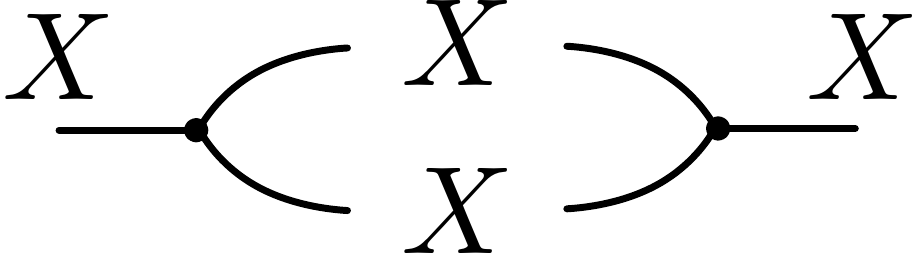}
\end{center}
With these maps, we can represent a cross timestep constraint with diagrams of the following form.
\begin{center}
    \includegraphics[scale=.3]{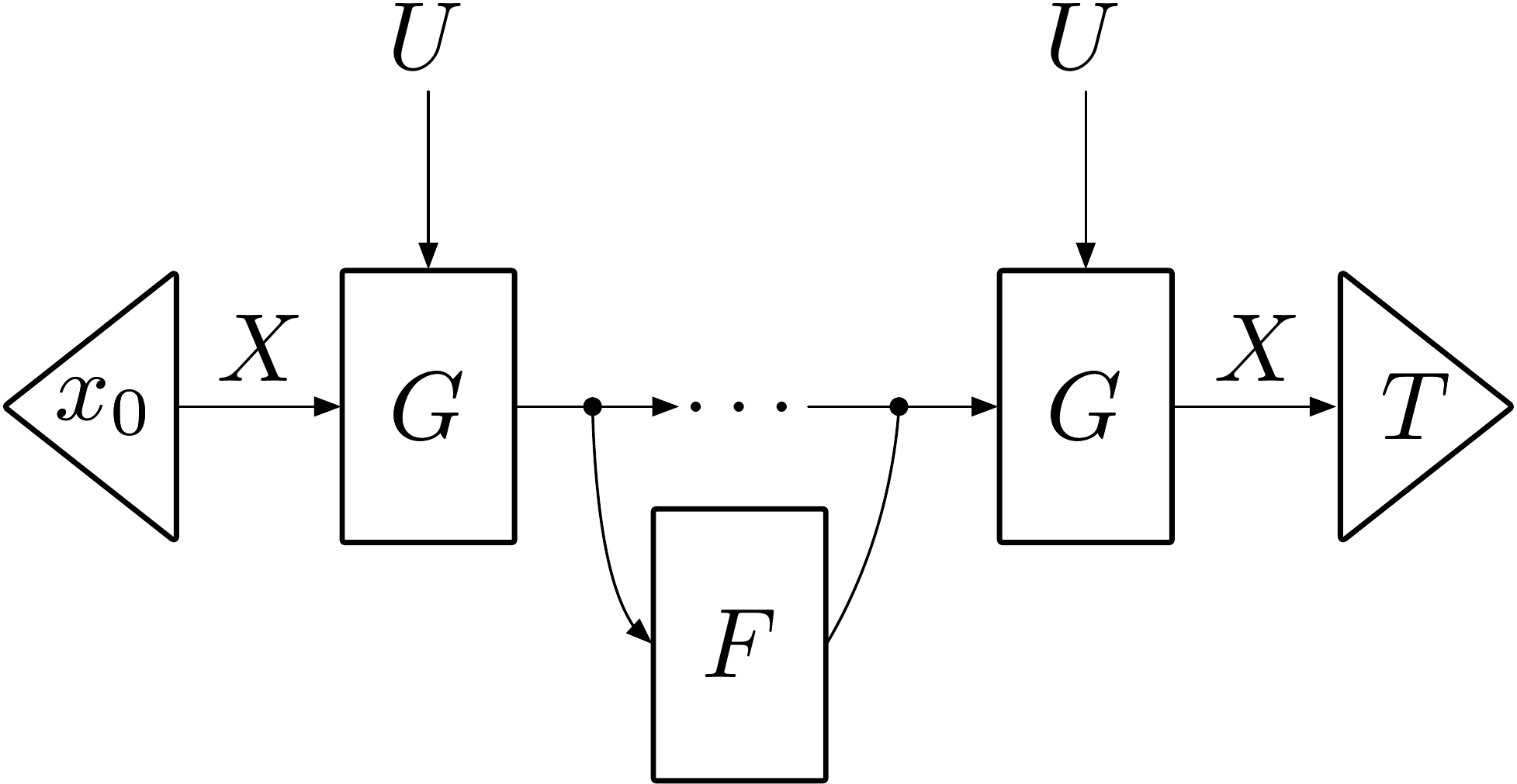}
\end{center}
In this example, the state after 1 timestep and after $N-1$ timesteps are related by the bifunction $F\maps X\hto X$, which can encode an arbitrary cost or constraint term to couple these two state variables.
\vspace{-2ex}
\subsection{Utilizing Hierarchical Composition}

The final type of composition available to us in $\Para(\Conv)$ is hierarchical composition. This gives us a way of modularly building up large problems by nesting collections of subproblems with the correct inputs and outputs within a box in a wiring diagram. To illustrate this, we can consider ``zooming in" on a single $G$ box, and building this morphism from even simpler components.

Recall that a single timestep subproblem is composed of an objective function $\ell$, an inequality constraint function $g$, and an equality constraint function $f_d$. The following diagram shows how we can fill a box representing a single timestep with these component subproblems (where we abuse notation to treat the constraint functions as their associated indicator bifunctions).
\begin{center}
    \includegraphics[scale=.3]{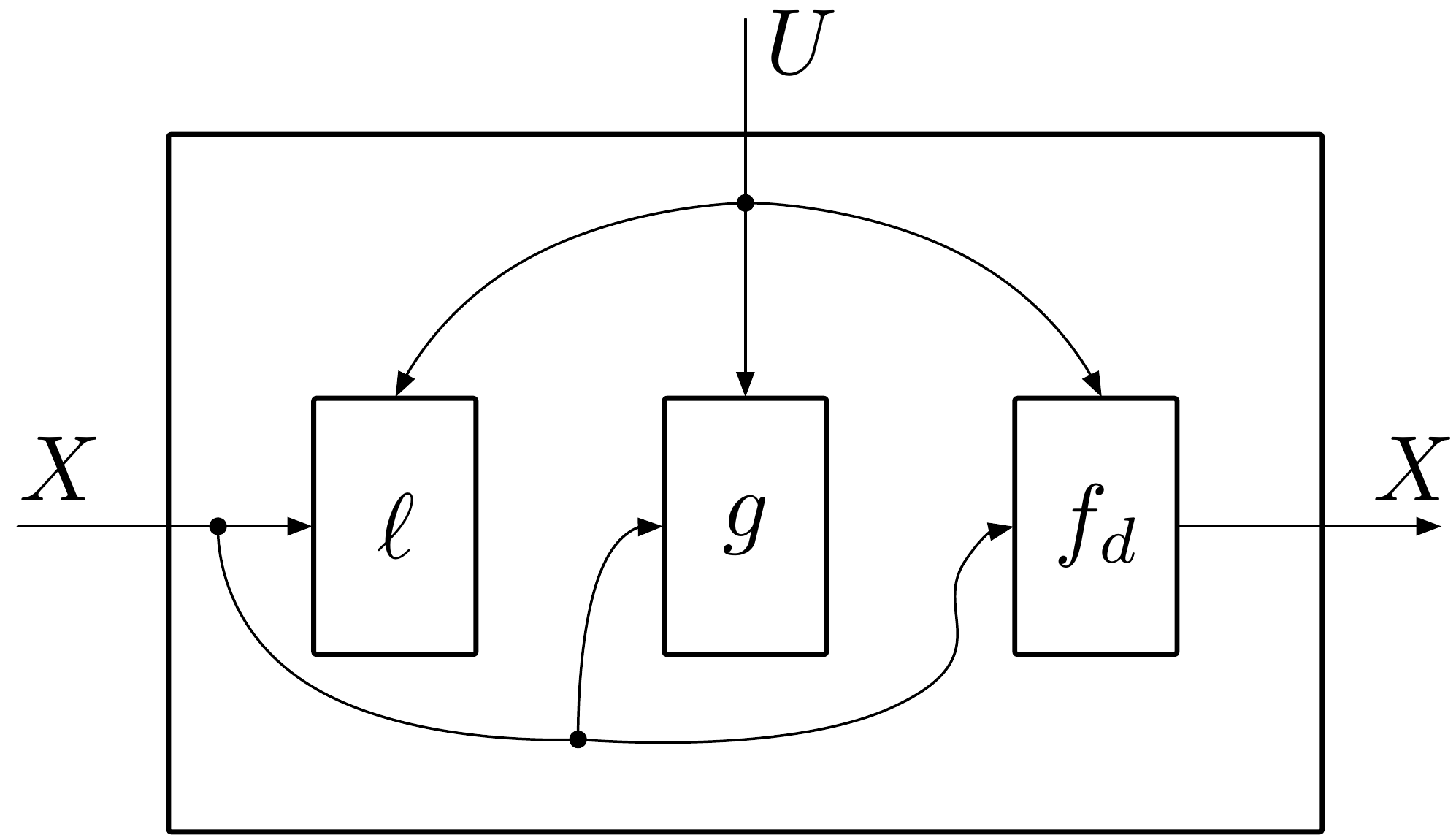}
\end{center}
This allows us to package up multiple morphisms into a modular, reusable component.
\vspace{-2ex}
\section{Implementation}
\label{sec:Impl}

\begin{listing}
\begin{minted}
[
bgcolor=mygray,
fontsize=\footnotesize
]
{julia}
# Declare problem constants.
A, B = [0 1; .01 0], [0; 1]
Q, R = 5*I(2), 3
x0 = [3.0, 1.0]

# Specify system dynamics and one-step cost
# and constraint functions.
cost(u,x) = quadform(x,Q) + R*square(u)
dynamics(u,x) = A*x + B*u
constraints(u,x) = [
    u <= 1, u >= -1,
    x[1] <= 3, x[1] >= -3,
    x[2] <= 2, x[2] >= -2
]

# Create the one-step bifunction.
one_step = one_step_bifunction(dim_x,dim_u,cost, 
    constraints, dynamics)

N = 10 # prediction horizon
MPC_bifunc = compose(repeat([one_step], N))

# Create variables to store control inputs
# and final state achieved.
us = [Variable(1) for i in 1:N-1]
xN = Variable(2)

# Convert to a Convex.jl problem and solve.
MPC_prob = make_problem(MPC_bifunc, us, x0, xN)
solve!(MPC_prob, SCS.Optimizer)
\end{minted}
\vspace{-2ex}
\caption{This code snippet shows how Example \ref{ex:running} is implemented using AlgebraicControl.jl. Note that the dynamics, cost, and constraints are specified in terms of a single time-step. Any prediction horizon can then be chosen for the problem and bifunction composition handles the required book keeping to build up the problem.}
\label{l:good_impl}

\end{listing}

One of the many benefits of framing a domain categorically is that correct and clean implementations should follow easily from the proof that the domain has the structure of a certain type of category. Catlab.jl is a software package written in the Julia programming language with the goal of realizing this benefit \cite{catlab}. Specifically, Catlab allows the programmer to create instances of categorical abstractions such as symmetric monoidal categories, and then provides tools to automatically compile morphisms expressed symbolically or as string diagrams in those categories to Julia code.

To apply this methodology to our framework, we instantiate the theory of symmetric monoidal categories using Convex.jl problems as a software representation for convex bifunctions. Convex.jl is a Julia implementation of disciplined convex programming \cite{dcp}\cite{convexjl} . We then simply had to implement the identity, composition, and monoidal product operations described in this paper, and apply the $\Para$ construction, supplied for us by Catlab, to obtain the desired results. To summarize, our library, which we call AlgebraicControl.jl, has a core API consisting of the following three functions:

\begin{enumerate}
    \item \texttt{one\_step\_bifunction(X, U, l, g, f\_d)}
    \item \texttt{compose(bifunctions\dots)}
    \item \texttt{make\_problem(F, us, x0, xN)}
\end{enumerate}
\noindent The function \texttt{one\_step\_bifunction} takes the state space~$X$, control space~$U$, loss function~$\ell$, constraint function~$g$, and dynamics $f_d$, and then it  constructs the one-step bifunction associated with this data as in Definition \ref{def:onestep}. Next, the function \texttt{compose} takes a list of bifunctions and constructs their composite. When the argument to compose is $N$ copies of a one-step bifunction associated with an MPC subproblem, Theorem \ref{thm:nfoldcomp} ensures that the resultant composite correctly encodes the corresponding MPC optimization problem for a prediction horizon of $N$. Finally, the function \texttt{make\_problem} takes a parameterized bifunction $F$ along with a list of control variables and initial and final state variables and exports this data to a Convex.jl problem. 
The resultant optimization problem can then be readily solved by any solver supported by Convex.jl. 

\vspace{-1ex}
\begin{example}
For a specific choice of $A$, $B$, $Q$, and $R$, Listing \ref{l:good_impl} shows how AlgebraicControl.jl automates the construction of Example \ref{ex:running} and the translation into a Convex.jl problem. 
This framework uses category theory, namely Proposition \ref{prop:genconv} and Theorem \ref{thm:nfoldcomp}, to
allow control engineers to focus on the high-level components of the problem, like accurately modeling the system dynamics and choosing appropriate costs and constraints. 
This is because bifunction composition handles the book-keeping for managing variables at each time step and ensuring the subproblems are correctly coupled.
\end{example}

\vspace{-2ex}
\section{Conclusion and Future Work}
We have developed a categorical framework for formally linking model predictive control problem formulations with convex optimization problem formulations that captures the categorical structure of MPC problems. 
Future work will address the construction of specialized convex optimization routines that can exploit the categorical structure of MPC problems as identified in this paper. Finally, although we focused specifically on MPC problems for linear dynamical systems, there is nothing in principle which prevents this framework from being extended to the non-linear case. This is another direction for future work.


\bibliographystyle{unsrt}
\bibliography{sample}

\end{document}